\documentclass[11pt]{article}

\usepackage{amsfonts,amssymb,amstext,amsmath,latexsym,color,epsfig}
\newtheorem{theorem}{Theorem}
\newtheorem{lemma}{Lemma}

\newenvironment{proof}
      {\medskip\noindent{\bf Proof:}\hspace{1mm}}
      {\hfill$\Box$\medskip}

\input{epsf}

\makeatletter
\def\Ddots{\mathinner{\mkern1mu\raise\p@
\vbox{\kern7\p@\hbox{.}}\mkern2mu
\raise4\p@\hbox{.}\mkern2mu\raise7\p@\hbox{.}\mkern1mu}}
\makeatother

\title{\vspace{-0.7cm}Graphs with few paths of prescribed length between any two vertices}
\author{David Conlon\thanks{Mathematical Institute, Oxford OX2 6GG, United Kingdom. E-mail: {\tt david.conlon@maths.ox.ac.uk}. Research supported by a Royal Society University Research Fellowship.}}
\date{}
\begin{document}
\maketitle

\begin{abstract}
We use a variant of Bukh's random algebraic method to show that for every natural number $k \geq 2$ there exists a natural number $\ell$ such that, for every $n$, there is a graph with $n$ vertices and $\Omega_k(n^{1 + 1/k})$ edges with at most $\ell$ paths of length $k$ between any two vertices. A result of Faudree and Simonovits shows that the bound on the number of edges is tight up to the implied constant.
\end{abstract}

\section{Introduction}

Given a graph $H$, let $\textrm{ex}(n,H)$ be the maximum number of edges in an $H$-free graph on $n$ vertices. The classic Erd\H{o}s--Stone--Simonovits theorem \cite{ES66, ES46} gives a satisfactory first estimate for this function, showing that
\[\textrm{ex}(n, H) = \left(1 - \frac{1}{\chi(H)-1} + o(1)\right) \binom{n}{2},\]
where $\chi(H)$ is the chromatic number of $H$. For bipartite $H$, this gives the bound $\textrm{ex}(n, H) = o(n^2)$. While more precise estimates are known, a number of notoriously difficult open problems remain. 

The most intensively studied case is when $H = K_{s,t}$, the complete bipartite graph with parts of order $s$ and $t$. In this case, a famous result of  K\H{o}v\'ari, S\'os and Tur\'an \cite{KST54} shows that $\textrm{ex}(n, K_{s,t}) = O_{s,t}(n^{2 - 1/s})$ whenever $s \leq t$. This bound was shown to be tight for $s = 2$ by Klein~\cite{E38} (see also~\cite{ERS66}) and for $s = 3$ by Brown \cite{B66}. For higher values of $s$, it is only known that the bound is tight when $t$ is sufficiently large in terms of $s$. This was first shown by Koll\'ar, R\'onyai and Szab\'o \cite{KRS96}, though the construction was improved slightly by Alon, R\'onyai and Szab\'o \cite{ARS99}, who showed that there are graphs with $n$ vertices and $\Omega_s(n^{2 - 1/s})$ edges containing no copy of $K_{s, t}$ with $t = (s-1)! + 1$. An alternative construction, with a slightly weaker bound on $t$, was also given by Blagojevi\'c, Bukh and Karasev \cite{BBK13}.

Very recently, Bukh \cite{B14} found a simple, elegant method for showing that the K\H{o}v\'ari--S\'os--Tur\'an bound is tight for $t$ sufficiently large in terms of $s$, fusing the algebraic techniques used in all previous constructions with an application of the probabilistic method. In this paper, we adapt his method to make progress on an equally stubborn problem.

When $H = C_{2k}$, the cycle with $2k$ vertices, a result of Erd\H{o}s (see \cite{BS74}) shows that $\textrm{ex}(n, C_{2k}) = O_k(n^{1 + 1/k})$. Since $C_4 = K_{2,2}$, the result of Klein mentioned above shows that this bound is tight for $k = 2$. For $k = 3$ and $5$, constructions matching the upper bound were found by Benson~\cite{Be66} and Singleton~\cite{S66} and later by Wenger~\cite{W91}, Lazebnik and Ustimenko~\cite{LU95} and Mellinger and Mubayi~\cite{MM05}. For general $k$, the best known lower bound on $\textrm{ex}(n, C_{2k})$ is due to Lazebnik, Ustimenko and Woldar~\cite{LUW95}, but does not match the upper bound.

If we let $\theta_{k, \ell}$ be the graph consisting of $\ell$ internally disjoint paths of length $k$, each with the same endpoints, we see that $\theta_{k, 2} = C_{2k}$ and so the problem of determining $\textrm{ex}(n, \theta_{k, \ell})$ generalises the problem of determining $\textrm{ex}(n, C_{2k})$. This problem was first studied by Faudree and Simonovits \cite{FS83}, who showed that $\textrm{ex}(n, \theta_{k, \ell}) = O_{k, \ell}(n^{1 + 1/k})$ for all $k$ and $\ell \geq 2$. The lower bounds for $\textrm{ex}(n, C_{2k})$ show that this bound is tight when $k = 2, 3$ or $5$. Additionally, a result of Verstra\"ete and Williford \cite{VW19} shows that it is tight for $k = 4$ and $\ell \geq 3$. In this paper, we generalise these results, showing that the upper bound is tight for every $k$, provided that $\ell$ is sufficiently large. This may be seen as an analogue of the Koll\'ar--R\'onyai--Szab\'o result for the cycle-free problem.

\begin{theorem} \label{main}
For any natural number $k \geq 2$, there exists a natural number $\ell$ such that
\[\mathrm{ex}(n, \theta_{k, \ell}) = \Omega_k(n^{1 + 1/k}).\]
\end{theorem}

More generally, we will show that for any natural number $k \geq 2$ there exists a natural number $\ell$ such that, for every $n$, there is a graph with $n$ vertices and $\Omega_k(n^{1 + 1/k})$ edges with at most $\ell$ (not necessarily disjoint) paths of length $k$ between any two vertices. Though heuristics suggest that it might be possible to take $\ell = k^{O(k^2)}$, our method falls well short of this. We have therefore made no systematic attempt to optimise (or even compute) the value of $\ell$ arising from our proof. Throughout the paper, we will use standard asymptotic notation, with subscripts indicating that the implied constants depend on these parameters.

\section{Preliminaries}

Let $q$ be a prime power and let $\mathbb{F}_q$ be the finite field of order $q$. We will consider polynomials in $t$ variables over $\mathbb{F}_q$, writing any such polynomial as $f(X)$, where $X = (X_1, \dots, X_t)$. We let $\mathcal{P}_d$ be the set of polynomials in $X$ of degree at most $d$, that is, the set of linear combinations over $\mathbb{F}_q$ of monomials of the form $X_1^{a_1} \cdots X_t^{a_t}$ with $\sum_{i=1}^t a_i \leq d$. By a random polynomial, we just mean a polynomial chosen uniformly from the set $\mathcal{P}_d$. One may produce such a random polynomial by choosing the coefficients of the monomials above to be random elements of $\mathbb{F}_q$.

We will now determine the probability that a randomly chosen polynomial from $\mathcal{P}_d$ passes through a given set of points. For one point, the probability is $1/q$, as shown by the following simple result of Bukh~\cite{B14}. We include the proof for completeness.

\begin{lemma} \label{edgeexp}
If $f$ is a random polynomial from $\mathcal{P}_d$, then, for any fixed $x \in \mathbb{F}_q^t$, 
$$\mathbb{P}[f(x) = 0] = 1/q.$$ 
\end{lemma}

\begin{proof}
Let $\mathcal{Q}_d = \{f \in \mathcal{P}_d: f(0) = 0\}$, that is, the collection of polynomials in $\mathcal{P}_d$ with zero constant term. Since every $f \in \mathcal{P}_d$ can be written as $g + h$ where $g \in \mathcal{Q}_d$ and $h$ is a constant, we may sample a random element of $\mathcal{P}_d$ by adding a random element $g$ of $\mathcal{Q}_d$ and a random element $h$ of $\mathbb{F}_q$. Since, for any fixed choice of $g$, there is only one choice out of $q$ for $h$ such that $f(x) = 0$, the result follows.
\end{proof}

The following result shows that once $q$ and $d$ are sufficiently large, the probability that a randomly chosen polynomial from $\mathcal{P}_d$ contains each of $m$ distinct points is exactly $1/q^{m}$. That is, the events are independent. This observation is again due to Bukh and corresponds to Lemma $4$ of his paper~\cite{B14}, though we state and prove it in greater generality.

\begin{lemma} \label{graphexp}
Suppose that $q > \binom{m}{2}$ and $d \geq m - 1$. Then, if $f$ is a random polynomial from $\mathcal{P}_d$ and $x_1, \dots, x_m$ are $m$ distinct points in $\mathbb{F}_q^t$,
$$\mathbb{P}[f(x_i) = 0 \mbox{ for all } i = 1, \dots, m] = 1/q^{m}.$$ 
\end{lemma}

\begin{proof}
Let $x_i = (x_{i,1}, \dots, x_{i,t})$ for each $i = 1, \dots, m$. We will choose elements $a_2, \dots, a_t  \in \mathbb{F}_q$ such that $x_{i,1} + \sum_{j=2}^t a_j x_{i, j}$ is distinct for all $i = 1, \dots, m$. To see that this is possible, note that there are exactly $\binom{m}{2}$ equations 
\[x_{i,1} + \sum_{j=2}^t a_j x_{i, j} = x_{i',1} + \sum_{j=2}^t a_j x_{i', j},\] 
each with at most $q^{t-2}$ solutions $(a_2, \dots, a_t)$. Therefore, since the total number of choices for $(a_2, \dots, a_t)$ is $q^{t-1}$ and $q^{t-1} > q^{t-2} \binom{m}{2}$, we can make an appropriate choice.

We now consider $\mathcal{P}'_d$, the set of polynomials of degree at most $d$ in $Z$, where $Z_1 = X_1 + \sum_{j=2}^t a_j X_j$ and $Z_j = X_j$ for all $2 \leq j \leq t$. Since this change of variables is an invertible linear map, $\mathcal{P}'_d$ is identical to $\mathcal{P}_d$, so it will suffice to show that a randomly chosen polynomial from $\mathcal{P}'_d$ passes through all of the points $z_1, \dots, z_m$ corresponding to $x_1, \dots, x_m$ with probability $q^{-m}$. Note that, by our choice above, $z_{i,1} \neq z_{i',1}$ for any $1 \leq i < i' \leq m$.

For any $f$ in $\mathcal{P}'_d$, we may write $f = g + h$, where $h$ contains all monomials of the form $Z_1^j$ for $j = 0, 1, \dots, m-1$ and $g$ contains all other monomials. For any fixed choice of $g$, there is exactly one choice of $h$ such that $f(z_i) = 0$ for all $i = 1, \dots, m$, namely, the unique polynomial of degree at most $m-1$ which takes the value $-g(z_i)$ at $z_{i,1}$ for all $i = 1, 2, \dots, m$, where uniqueness follows from the fact that the $z_{i,1}$ are distinct. Since there were $q^m$ possible choices for $h$, the result follows. 
\end{proof}

We also need to note some basic facts about affine varieties over finite fields. If we write $\overline{\mathbb{F}}_q$ for the algebraic closure of $\mathbb{F}_q$, a variety over $\overline{\mathbb{F}}_q$ is a set of the form
\[W = \{x \in \overline{\mathbb{F}}_q^t : f_1(x) = \dots = f_s(x) = 0\}\]
for some collection of polynomials $f_1, \dots, f_s : \overline{\mathbb{F}}_q^t \rightarrow \overline{\mathbb{F}}_q$. We say that $W$ is defined over $\mathbb{F}_q$ if the coefficients of these polynomials are in $\mathbb{F}_q$ and write $W(\mathbb{F}_q) = W \cap \mathbb{F}_q^t$. We say that $W$ has complexity at most $M$ if $s$, $t$ and the degrees of the $f_i$ are all bounded by $M$. Finally, we say that a variety is absolutely irreducible if it is irreducible over $\overline{\mathbb{F}}_q$, reserving the term irreducibility for irreducibility over $\mathbb{F}_q$ of varieties defined over $\mathbb{F}_q$.

The first result we will need is the Lang--Weil bound \cite{LW54} relating the dimension of a variety $W$ to the number of points in $W(\mathbb{F}_q)$. Recall that the dimension $\dim W$ of an affine variety $W$ is the maximum integer $d$ such that there exists a chain of absolutely irreducible subvarieties of $W$ of the form
\[\emptyset \subsetneq \{p\} \subsetneq W_1 \subsetneq W_2 \subsetneq \dots \subsetneq W_d \subset W,\]
where $p$ is a point. As originally stated, the Lang--Weil bound applies to projective varieties, but it is a simple matter to deduce the version for affine varieties given below.

\begin{lemma} \label{LW}
Suppose that $W$ is a variety over $\overline{\mathbb{F}}_q$ of complexity at most $M$. Then 
$$|W(\mathbb{F}_q)| = O_M(q^{\dim W}).$$
Moreover, if $W$ is defined over $\mathbb{F}_q$ and absolutely irreducible, then 
$$|W(\mathbb{F}_q)| = q^{\dim W}(1 + O_M(q^{-1/2})).$$
\end{lemma}

We will also need the following standard result from algebraic geometry (see, for example, \cite{B98}), which says that if $W$ is an irreducible affine variety over an algebraically closed field and $g$ is a polynomial whose zero set intersects $W$ then either $W$ is contained in the zero set of $g$ or its intersection with this zero set has smaller dimension.

\begin{lemma} \label{dimdrop}
Suppose that $W$ is an absolutely irreducible variety over $\overline{\mathbb{F}}_q$ and $\dim W \geq 1$. Then, for any polynomial $g: \overline{\mathbb{F}}_q^t \rightarrow \overline{\mathbb{F}}_q$, $W \subseteq \{x : g(x) = 0\}$ or $W \cap \{x : g(x) = 0\}$ is a variety of dimension less than $\dim W$.
\end{lemma}

The final ingredient we require is again essentially due to Bukh \cite{B14} and says that if $W$ is a variety which is defined over $\mathbb{F}_q$, then there is a finite collection of absolutely irreducible varieties $Y_1, \dots, Y_s$, each of which is defined over $\mathbb{F}_q$, such that $\cup_{i=1}^s Y_i(\mathbb{F}_q) = W(\mathbb{F}_q)$. Since it is less standard than the previous two lemmas, we will include the proof. We will repeatedly use the fact that a variety $X$ is defined over $\mathbb{F}_q$ if and only if it is fixed by the Frobenius automorphism taking any element $x$ to $x^q$. The non-trivial direction of this equivalence requires one to consider the reduced Gr\"obner basis of the ideal associated to the variety (see, for example, the proof of Corollary 4 in Tao's blog post on the Lang--Weil bound \cite{T12}).

\begin{lemma} \label{frob}
Suppose that $W$ is a variety over $\overline{\mathbb{F}}_q$ of complexity at most $M$ which is defined over $\mathbb{F}_q$. Then there are $O_M(1)$ absolutely irreducible varieties $Y_1, \dots, Y_s$, each of which is defined over $\mathbb{F}_q$ and has complexity $O_M(1)$, such that $\cup_{i=1}^s Y_i(\mathbb{F}_q) = W(\mathbb{F}_q)$.
\end{lemma}

\begin{proof}
We begin by splitting $W$ into irreducible components, noting, for instance, by Lemma 11 of~\cite{T10}, that the number and complexity of these components, each of which is defined over $\mathbb{F}_q$, is bounded by a function of the complexity of $W$. Let $X$ be an irreducible component of $W$. If $X$ is absolutely irreducible, we set it aside as one of our $Y_i$. If $X$ is not absolutely irreducible, we let $X_1, \dots, X_r$ be the absolutely irreducible components of $X$, noting again that the number and complexity of these components is bounded by a function of the complexity of $W$. Since $X$ is defined over $\mathbb{F}_q$, the Frobenius automorphism $x \mapsto x^q$ acts on its components, permuting $X_1, \dots, X_r$. Moreover, this action is transitive. This is because the union of the sets in any orbit is fixed by the Frobenius automorphism and so defined over $\mathbb{F}_q$. Therefore, if the action were not transitive, $X$ would be reducible over $\mathbb{F}_q$. By transitivity, we have that $X(\mathbb{F}_q) \subseteq X_i$ for all $i = 1, \dots, r$. But then $X(\mathbb{F}_q) \subseteq X_1 \cap \dots \cap X_r$, which, since $\dim X_1 \cap X_2 < \dim X_1 \leq \dim W$, is a variety of lower dimension than $W$ that is defined over $\mathbb{F}_q$ and has complexity bounded by a function of the complexity of $W$. We may now repeat the entire procedure with $X_1 \cap \dots \cap X_r$. Since the dimension is bounded below by zero,  this iteration must eventually terminate.
\end{proof}

\section{The construction}

Let $t = r = 2k$, $d = kr$, $N = q^k$ and suppose that $q$ is sufficiently large. Let $f_1, \dots, f_{k-1} : \mathbb{F}_q^k \times \mathbb{F}_q^k \rightarrow \mathbb{F}_q$ be independent random polynomials in $\mathcal{P}_d$. We consider the bipartite graph $G$ between two copies $U$ and $V$ of $\mathbb{F}_q^k$, each of order $N = q^k$, where $(u, v)$ is an edge of $G$ if and only if 
$$f_1(u, v) = \dots = f_{k-1}(u,v) = 0.$$ 
Since $f_1, \dots, f_{k-1}$ were chosen independently, Lemma~\ref{edgeexp} tells us that the probability a given edge $(u,v)$ is in $G$ is $q^{-(k-1)}$. Therefore, the expected number of edges in $G$ is $q^{-(k-1)} N^2 = N^{1 + 1/k}$.

Suppose now that $w_1$ and $w_2$ are two fixed vertices in $G$ and let $S$ be the set of paths of length $k$ between them. We will be interested in estimating the $r$-th moment of $|S|$. To begin, we note that $|S|^r$ counts the number of ordered collections of $r$ (possibly overlapping or identical) paths of length $k$ in $G$ between $w_1$ and $w_2$. Since the total number of edges $m$ in any given collection of $r$ paths is at most $k r = d$ and $q$ is sufficiently large, Lemma~\ref{graphexp} tells us that the probability that particular collection of paths is in $G$ is $q^{-(k-1)m}$, where we again used the fact that $f_1, \dots, f_{k-1}$ are chosen independently. 

Within the complete bipartite graph between $U$ and $V$, let $P_{r, m}$ be the number of ordered collections of $r$ paths, each of length $k$, from $w_1$ to $w_2$ whose union has $m$ edges. Then
\[\mathbb{E}[|S|^r] = \sum_{m = 1}^{kr} P_{r,m} q^{-(k-1)m}.\]
To estimate $P_{r, m}$ and hence $\mathbb{E}[|S|^r]$, we will show that if the union of $r$ paths, each of length $k$, has $m$ edges then the number of vertices $n$ other than $w_1$ and $w_2$ satisfies $kn \leq (k-1)m$. To see this, suppose that $p_1, \dots, p_r$ are paths of length $k$ whose union has $m$ edges. We consider the paths in sequence, letting $n_i$ be the number of vertices and $m_i$ the number of edges in $p_i \setminus p_1 \cup \dots \cup p_{i-1}$. If $n_i \neq 0$, we have 
$$m_i \geq n_i + 1 \geq \frac{n_i + 1}{n_i} n_i \geq \frac{k}{k-1} n_i.$$ 
Summing over all $i$ for which $n_i \neq 0$ gives the required inequality. Therefore, $P_{r,m} = O_k(N^{(k-1)m/k})$ and
\[\mathbb{E}[|S|^r] = \sum_{m = 1}^{kr} P_{r,m} q^{-(k-1)m} = \sum_{m=1}^{kr} O_k(N^{(k-1)m/k}) q^{-(k-1)m} = \sum_{m=1}^{kr} O_k(q^{(k-1)m}) q^{-(k-1)m} = O_k(1).\]
By Markov's inequality, we may conclude that, for any positive $s$,
\[\mathbb{P}[|S| \geq s] = \mathbb{P}[|S|^r \geq s^r] \leq \frac{\mathbb{E}[|S|^r]}{s^r} = \frac{O_k(1)}{s^r}.\]
We now note that the set of paths $S$ is a subset of $T(\mathbb{F}_q)$, where
\[T = \{(x_1, \dots, x_{k-1}) \in \overline{\mathbb{F}}_q^{k(k-1)} : f_i(w_1, x_1) = f_i(x_2, x_1) = \dots = f_i(x_{k-1}, w_2) = 0 \mbox{ for } i = 1, \dots, k-1\}.\]
Note that, depending on which sides of the bipartition contain $w_1$ and $w_2$, the order of the two variables in many of these equations may need to be reversed. However, this makes little difference to what follows, so we will assume that the orders are as given above. 

If $T(\mathbb{F}_q)$ were equal to $S$, we could apply Lemma 5 from \cite{B14} to show that $S$ is either bounded by a constant or quite large and then use the corollary of Markov's inequality proved above to show that there are very few pairs $(w_1, w_2)$ for which $S$ is large. Unfortunately, $T(\mathbb{F}_q)$ may contain degenerate walks as well as the paths we are interested in, so we must somehow take these into account.

If we write $x_0$ for $w_1$ and $x_k$ for $w_2$, we see that if $T(\mathbb{F}_q)$ contains a degenerate walk $w_1, x_1, \dots, x_{k-1}, w_2$, then it must be the case that $x_a = x_b$ for some $a$ and $b$ with $0 \leq a < b \leq k$ and $b - a$ even. This naturally leads us to consider the collections of sets
\[T_{ab} = T \cap \{(x_1, \dots, x_{k-1}) : x_a = x_b\}\] 
for all $a$ and $b$ with $0 \leq a < b \leq k$ and $b - a$ even.

Since $T$ is defined over $\mathbb{F}_q$ and has complexity bounded in terms of $k$, Lemma~\ref{frob} tells us that there are $O_k(1)$ absolutely irreducible varieties $Y_1, \dots, Y_s$, each of which is defined over $\mathbb{F}_q$ and has complexity $O_k(1)$, such that $\cup_{i=1}^s Y_i(\mathbb{F}_q) = T(\mathbb{F}_q)$. If $\dim Y_i \geq 1$, Lemma~\ref{dimdrop} tells us that either there exist $a$ and $b$ such that $Y_i \subseteq T_{ab}$ or the dimension of $Y_i \cap T_{ab}$ is smaller than the dimension of $Y_i$ for all $a$ and $b$. If $Y_i \subseteq T_{ab}$ for some $a$ and $b$, the component does not contain any non-degenerate paths and may be removed from consideration. If instead the dimension of $Y_i \cap T_{ab}$ is smaller than the dimension of $Y_i$ for all $a$ and $b$, the Lang--Weil bound, Lemma~\ref{LW}, tells us that for $q$ sufficiently large
\[|S| \geq |Y_i(\mathbb{F}_q)| - \sum_{a,b} |Y_i \cap T_{ab}(\mathbb{F}_q)| \geq q^{\dim Y_i} - O_k(q^{\dim Y_i - \frac{1}{2}}) - O_k(q^{\dim Y_i - 1}) \geq \frac{q}{2}.\]
On the other hand, if $\dim Y_i = 0$ for every $Y_i$ which is not contained in some $T_{ab}$, Lemma~\ref{LW} tells us that $|S| \leq \sum |Y_i(\mathbb{F}_q)| =O_k(1)$, where the sum is taken over all $i$ for which $\dim Y_i = 0$.

Putting everything together, we see that that there exists a constant $c_k$, depending only on $k$, such that either $|S| \leq c_k$ or $|S| \geq q/2$. Therefore, by the consequence of Markov's inequality noted earlier,
\[\mathbb{P}[|S| > c_k] = \mathbb{P}[|S| \geq q/2] \leq \frac{O_k(1)}{(q/2)^r}.\]
We call a pair of vertices $(w_1, w_2)$ bad if there are more than $\ell = c_k$ paths between them. If we let $B$ be the random variable counting the number of bad pairs, we have, since $r = 2k$,
\[\mathbb{E}[B] \leq 2N^2 \cdot \frac{O_k(1)}{(q/2)^r} = O_k(q^{2k - r}) = O_k(1).\]
We now remove a vertex from each bad pair to form a new graph $G'$. Since each vertex has degree at most $N$, the total number of edges removed is at most $B N$. Hence, the expected number of edges is
\[N^{1 + 1/k} - \mathbb{E}[B] N = \Omega_k(N^{1 + 1/k}).\]
Therefore, there is a graph with at most $2N$ vertices and $\Omega_k(N^{1 + 1/k})$ edges such that no two vertices have more than $\ell = c_k$ paths of length $k$ between them. As stated, this result only holds when $q$ is a prime power and $N = q^k$. However, it is a simple matter to use Bertrand's postulate to show that the same conclusion holds for all $N$.

\vspace{3mm}
\noindent
{\bf Acknowledgements.} I would like to thank Boris Bukh, Gal Kronenberg, Rudi Mrazovi\'c and Lisa Sauermann for a number of valuable comments on an earlier draft of this paper. I would also like to thank the anonymous referee for their considered review.

\end{document}